\documentclass[11pt,draft,reqno]{amsart}

\usepackage{amssymb} 
\usepackage{dsfont} 
\usepackage{mathrsfs} 
\usepackage{stmaryrd} 
\usepackage{color}
\usepackage{xcolor}

\usepackage{soul}

\DeclareMathAlphabet{\mathpzc}{OT1}{pzc}{m}{it} 

\newtheorem{Thm}{Theorem}[section]

\newtheorem{Lem}{Lemma}[section]
\newtheorem{Prop}{Proposition}[section]

\newtheorem{Def}{Definition}[section]
\newtheorem{Pb}{Problem}[section]

\theoremstyle{definition}

\theoremstyle{definition}

\makeatletter
\@addtoreset{equation}{section}
\makeatother

\newcommand\function{\longrightarrow} 

\newcommand\en{\mathbb{N}} 
\newcommand\re{\mathbb{R}} 

\providecommand{\clint}[1]{\hspace{0.045ex}\left[#1\right]} 

\renewcommand\H{\mathcal{H}} 
\newcommand\duality[2]{\langle #1,#2 \rangle} 
\newcommand\norm[2]{\Vert #1\Vert_{#2}} 

\newcommand\vartot[1]{\!\left\bracevert\!\! #1 \!\!\right\bracevert\!} 

\newcommand{\Czero}{{\textsl{C}\hspace{0.18ex}}} 
\newcommand{\BV}{{\textsl{BV}\hspace{0.17ex}}} 

\DeclareMathOperator{\De}{D\!} 

\newcommand{\Pl}{{\mathsf{P}}} 
\newcommand{\Q}{{\mathsf{Q}}} 

\providecommand{\clsxint}[1]{\hspace{0.1ex}\left[#1\right[\hspace{0.15ex}} 
\providecommand{\cldxint}[1]{\hspace{0.15ex}\left]#1\right]} 
\providecommand{\opint}[1]{\hspace{0.15ex}\left]#1\right[\hspace{0.15ex}} 

\newcommand\X{\textsl{X}\hspace{0.21ex}} 
\renewcommand\d{\textsl{d}} 

\DeclareMathOperator{\Int}{int} 

\DeclareMathOperator{\Lipcost}{Lip} 
\newcommand{\Lip}{{\textsl{Lip}\hspace{0.15ex}}} 

\newcommand{\Reg}{{\textsl{Reg}\hspace{0.17ex}}} 

\DeclareMathOperator{\pV}{V} 
\DeclareMathOperator{\V}{V} 

\renewcommand\S{\mathcal{S}\hspace{0.21ex}} 

\newcommand\void{\varnothing} 

\newcommand\Z{\mathcal{Z}} 

\newcommand\K{\mathcal{K}} 

\DeclareMathOperator{\Proj}{Proj} 

\newcommand{\AC}{{\textsl{AC}\hspace{0.17ex}}} 

\newcommand\indicator{\mathds{1}} 

\newcommand{\eps}{\varepsilon} 

\newcommand{\borel}{\mathscr{B}} 
\renewcommand{\L}{{\textsl{L}\hspace{0.17ex}}} 
\newcommand\leb{\mathpzc{L}} 
\DeclareMathOperator{\de}{d \! \hspace{0.2ex}} 
\newcommand{\Step}{{\textsl{St}\hspace{0.17ex}}} 

\newcommand{\convergedeb}{\rightharpoonup} 

\newcommand{\W}{{\textsl{W}\hspace{0.17ex}}} 

\newcommand{\ftilde}{\widetilde{f}} 
\newcommand{\utilde}{\widetilde{u}} 
\newcommand{\vtilde}{\widetilde{v}} 

\newcommand{\Stop}{{\mathsf{S}}} 

\renewcommand\sp{\hspace{3.1ex}} 

\definecolor{blu}{rgb}{0.1,0.1,1}

\definecolor{green}{rgb}{0.0, 0.5, 0.0}

\definecolor{marr}{rgb}{0.63, 0.47, 0.35}

\begin{document}


\title[Non-convex play operator]{Continuity of the non-convex play operator  \\ in the space of rectifiable curves}

\author{Jana Kopfov\'a, Vincenzo Recupero}
\thanks{The second author is a member of GNAMPA-INdAM}

\address{\textbf{Jana Kopfov\'a} \\
Mathematical Institute of the Silesian University\\ 
Na Rybn\' i\v cku 1, CZ-74601 Opava\\ Czech Republic.
    \newline
        {\rm E-mail address:}
        {\tt Jana.Kopfova@math.slu.cz}}     
          
\address{\textbf{Vincenzo Recupero} \\
        Dipartimento di Scienze Matematiche \\ 
        Politecnico di Torino \\
        Corso Duca degli Abruzzi 24 \\ 
        I-10129 Torino \\ 
        Italy. \newline
        {\rm E-mail address:}
        {\tt vincenzo.recupero@polito.it}}




\begin{abstract}
In this paper we prove that the vector play operator with a uniformly prox-regular characteristic set of constraints is continuous with respect to the $\BV$-norm and to the $\BV$-strict metric in the space of continuous functions of bounded variation. We do not assume any further regularity of the characteristic set. We also prove that the non-convex play operator is rate independent.
\end{abstract}


\maketitle


\thispagestyle{empty}


\section{Introduction}

Several phenomena in elasto-plasticity, ferromagnetism, and phase transitions are modeled by the following evolution variational inequality in a real Hilbert space $\H$ with the inner product 
$\duality{\cdot}{\cdot}$:
\begin{eqnarray}\label{var in-intro}
 \duality{z - u(t) + y(t)}{y'(t)} \le 0  && \forall z \in \Z, \quad t \in \clint{0,T},\\
\label{constraint}
u(t) - y(t) \in \Z && \forall t \in \clint{0,T}.
\end{eqnarray}
Here $u : \clint{0,T} \function \H$ is a given ``input'' function, $T > 0$ being the final time of evolution, and $y  : \clint{0,T} \function \H$ is the unknown function, $y'$ being its derivative. It is  assumed that the set $\Z$  in the constraint \eqref{constraint}  is a closed convex subset of $ \H,$ and it is usually called \emph{the characteristic set}. We refer to the monographs \cite{KrPo,Ma,Vi,BrSp,Kre96,MieRou15} for surveys on these physical models.
It is well-known (see, e.g., \cite{Kre96}), that if $u$ is  absolutely continuous, then
there exists a unique absolutely continuous solution $y$ of \eqref{var in-intro}-\eqref{constraint} together with the given initial condition
\begin{equation}\label{initial cond-intro}
  u(0) - y(0) = z_{0} \in \Z.
\end{equation}
If we set $\Pl(u,z_0) := y$ we have defined a solution operator $\Pl : \W^{1,1}(\clint{0,T};\H) \times \Z \function \W^{1,1}(\clint{0,T};\H)$ which is called the \emph{play operator}. Here 
$\W^{1,1}(\clint{0,T};\H)$ denotes the space of $\H$-valued Lipschitz continuous functions defined on $\clint{0,T}$ (precise definitions  will be given in Sections \ref{S:Preliminaries} and 
\ref{S:state main result}). 
An important feature of $\Pl$ is its \emph{rate independence}, i.e. 
\begin{equation}\label{rate ind}
  \Pl(u \circ \phi) = \Pl(u) \circ \phi
\end{equation}
whenever $\phi : \clint{0,T} \function \clint{0,T}$ is an increasing surjective Lipschitz continuous  reparametrization of time. The play operator can be extended to continuous
functions of bounded variation, i.e. to inputs $u \in \Czero(\clint{0,T};\H) \cap \BV(\clint{0,T};\H)$ (\cite{Kre96}). This can be done by reformulating \eqref{var in-intro} as an integral variation inequality:
\begin{equation}\label{play BV-integral inequality}
  \int_{0}^{T} \duality{z(t) - u(t) + y(t)}{\de y(t)} \le 0, \qquad \forall z \in \BV(\clint{0,T};\Z),
\end{equation}
where the integral can be interpreted as a Riemann-Stieltjes integral (see, e.g., \cite{Kre96}),
but also as a Lebesgue integral with respect to the differential measure $\De y$, the distributional 
derivative of $y$ (see \cite{Rec11} for the equivalence of the two formulations).
By \cite{Kre96} for every $u \in \Czero(\clint{0,T};\H) \cap \BV(\clint{0,T};\H)$ there exists a unique $y \in \Czero(\clint{0,T};\H) \cap \BV(\clint{0,T};\H)$ such that 
\eqref{play BV-integral inequality}, \eqref{constraint}, \eqref{initial cond-intro} hold. Therefore the play operator can be extended to the operator
$\Pl : \Czero(\clint{0,T};\H) \cap \BV(\clint{0,T};\H) \times \Z \to \Czero(\clint{0,T};\H) \cap \BV(\clint{0,T};\H).$
Its domain of definition is naturally endowed with the strong $\BV$-norm defined by
\begin{equation}\label{def BVnorm}
  \norm{u}{\BV} := \norm{u}{\infty} + \V(u,\clint{0,T}), \qquad u \in \BV(\clint{0,T};\H),
\end{equation}
where $\norm{u}{\infty}$ is the supremum norm of $u$ and $\V(u,\clint{0,T})$ is the total variation of $u$.
For absolutely continuous inputs the $\BV$-norm is exactly the standard $\W^{1,1}$-norm, and the continuity of $\Pl$ on $\W^{1,1}(0,T;\H)$ in this special case was proved in \cite{Kre91} for finite dimensional 
$\H$ and in \cite{Kre96} for separable Hilbert spaces. For such spaces $\H$, assuming $\Z$ has a smooth boundary, the $\BV$-norm continuity of $\Pl$ on $\BV(\clint{0,T};\H) \cap \Czero(\clint{0,T};\H)$ (respectively on $\BV(\clint{0,T};\H)$) was proved in \cite{BroKreSch04} (respectively in \cite{KrRo}). Under this additional regularity of $\Z$, in \cite{BroKreSch04, KrRo} it is also shown that $\Pl$ is locally Lipschitz continuous.
In \cite{KopRec16} we were able to drop the regularity of $\Z$ and we proved that $\Pl$ is $\BV$-norm continuous on $\BV(\clint{0,T};\H)$ for an arbitrary characteristic set $\Z$.

Another relevant topology in $\BV$ is the one induced by the so-called \emph{strict metric}, which is defined by
\begin{equation}\label{def strictBV}
  d_{s}(u,v) := \norm{u - v}{\infty} + |\V(u,\clint{0,T}) - \V(v,\clint{0,T})|, \qquad
  u, v \in \BV(\clint{0,T};\H),
\end{equation}
indeed every $u \in \BV(\clint{0,T};\H)$ can be approximated by a sequence $u_n \in \AC(\clint{0,T};\H)$ converging to $u$ in the strict metric. In \cite{Kre96} it is proved that $\Pl$ is continuous on
$\Czero(\clint{0,T};\H) \cap \BV(\clint{0,T};\H)$ with respect to the strict metric (shortly, ``strictly continuous''), provided $\Z$ has a smooth boundary. In 
\cite{Rec11} this regularity assumption is dropped and it is proved that $\Pl$ is continuous on
$\Czero(\clint{0,T};\H) \cap \BV(\clint{0,T};\H)$ with respect to the stric metric for every characteristic convex set $\Z$. In \cite{Rec11} it is also proved that in general $\Pl$ is not strictly continuous on the whole $\BV(\clint{0,T};\H)$.
For other results on the continuity properties of $\Pl$ we refer to \cite{Re4, Rec15a, KleRec16}.

Previous results are concerned with the case of a convex set $\Z$, but the characteristic set of constraints can be non-convex in some applications, e.g. in problems of crowd motion modeling (see \cite{Venel}). 

In the following we will restrict ourselves to 
uniform prox-regular sets - these are closed sets having a neighborhood where the projection exists and is unique. For the notion of prox-regularity we refer the reader to \cite{Fed, vial, ClaSteWol95, PolRocThi00, ColThi10}. Following e.g. \cite{ColMon03, KreMonRec22a, KreMonRec22b}, we see that 
the proper formulation of \eqref{play BV-integral inequality} in the case of a prox-regular set $\Z$ reads
\begin{equation}\label{BVnonconv}
   \int_{0}^T\duality{z(t) - u(t) + y(t)}{\de y(t)} 
   \le 
  \frac{1}{2r} \int_0^T \norm{z(t)-u(t)+y(t)}{}^2 \de V_y(t)\qquad \forall z \in \BV(\clint{0,T};\Z),
\end{equation}
where $V_y(t) = \V(y,\clint{0,t})$ for $t \in \clint{0,T}$ and $\norm{\cdot}{}$ is the norm in $\H$.
It is well-known (cf., e.g., \cite{EdmThi06} or \cite{KreMonRec22a}) that for every 
$u \in \Czero(\clint{0,T};\H) \cap \BV(\clint{0,T};\H)$  there exists a unique 
$y = \Pl(u,z_0) \in \Czero(\clint{0,T};\H) \cap \BV(\clint{0,T};\H)$ which satifies \eqref{BVnonconv}, \eqref{constraint}, \eqref{initial cond-intro}. Thus also in the non-convex case the solution operator
\[
\Pl : \Czero(\clint{0,T};\H) \cap \BV(\clint{0,T};\H) \times \Z \to \Czero(\clint{0,T};\H) \cap \BV(\clint{0,T};\H),
\]
of problem \eqref{BVnonconv}, \eqref{constraint}, \eqref{initial cond-intro} can be defined, which  we will call
\emph{non-convex play operator}.
In \cite{KreMonRec22b} it is proved that in $\W^{1,1}(\clint{0,T};\H)$ the operator $\Pl$ is continuous (and also local Lipschitz continuous) with respect to the strong $\BV$-norm under the assumption that $\Z$ satisfies a suitable regularity assumption, to be more precise it is required that $\Z$ is the sublevel set of a Lipschitz continuous function. In the present paper we prove that $\Pl$ is $\BV$-norm continuous on the larger space
$\Czero(\clint{0,T};\H) \cap \BV(\clint{0,T};\H)$ and for every characteristic prox-regular set $\Z$. We also prove that it is continuous with respect to the strict metric on the space of continuous functions of bounded variation. The technique of our proof consists in reducing the problem to the space of Lipschitz continuous functions, where the problem
is considerably easier. In order to perform the reduction we use the rate independence of $\Pl$, which, to the best of our knowledge is proved here for the first time in the non-convex case.
The question of the $\BV$-norm continuity on the whole space $\BV(\clint{0,T};\H)$ will be addressed in a future paper: in that case the presence of jumps makes the problem considerably more difficult and  the reparametrization method studied in \cite{Rec16a,Rec20} is needed.

The plan of the paper is the following: In Section 2 we recall the preliminaries needed to prove our results, which are stated in Section 3. In Section 4 we perform all the proofs.


\section{Preliminaries}\label{S:Preliminaries}

The set of integers greater  or equal to $1$ will be denoted by $\en$. 

\subsection{Prox-regular sets}

Throughout this paper we assume that
\begin{equation}\label{H-prel}
\begin{cases}
  \text{$\H$ is a real Hilbert space with the inner product 
  $\duality{x}{y}$}, \\
  \H \neq \{0\}, \\
  \norm{x}{} := \duality{x}{x}^{1/2} \qquad \text{for $x \in \H$}.
\end{cases}
\end{equation}
If $\S \subseteq \H$ and $x \in \H$ we set $\d_\S(x) :=  \inf\{\norm{x-s}{}\ :\ s \in \S\}$.

\begin{Def}
If $\K$ is a closed subset of $\H$, $\K \neq \void$, and $y \in \H$, we define the 
\emph{set of  projections of $y$ onto $\K$} by setting
\begin{equation}
 \Proj_\K(y) := \left\{x \in \K\ :\ \norm{x-y}{} = \inf_{z \in \K} \norm{z-y}{}\right\}
\end{equation}
and the \emph{(exterior) normal cone of $\K$ at $x$} by
\begin{equation}\label{normal cone}
  N_\K(x) := \{\lambda(y-x) \ :\ x \in \Proj_\K(y),\ y \in \H,\ \lambda \ge 0\}.
\end{equation}
\end{Def}

We recall the notion of prox-regularity (see \cite[Theorem 4.1-(d)]{ClaSteWol95}) which can also be called ``mild non-convexity''.

\begin{Def}
If $\K$ is a closed subset of $\H$ and if $r \in \opint{0,\infty}$, we say that $\K$ is 
\emph{$r$-prox-regular} if for every $y \in \{v \in \H\ :\ 0 < \d_{\K}(v) < r\}$ we have that 
$\Proj_\K(y) \neq \void$ and
\[
  x \in \Proj_\K\left(x+r\frac{y-x}{\norm{y-x}{}}\right), \qquad \forall x \in \Proj_\K(y).
\]
\end{Def}

It is well-known and easy to prove that if $x \in \Proj_\K(y_0)$ for some $y_0 \in \H$, then
$\Proj_\K(y) = \{x\}$ for every $y$ lying in the segment with endpoints $y_0$ and $x$. Thus it follows that if $\K$ is $r$-prox-regular for some $r > 0$, then $\Proj_\K(y)$ is a singleton for every 
$y \in \{v \in \H\ :\ 0 < \d_{\K}(v) < r\}$.

Prox-regularity can be characterized by means of a variational inequality, indeed in
\cite[Theorem 4.1]{PolRocThi00} and in \cite[Theorem 16]{ColThi10} one can find the proof of the following:

\begin{Thm}\label{charact proxreg}
Let $\K$ be a closed subset of $\H$ and let $r \in \opint{0,\infty}$. Then $\K$ is $r$-prox-regular if and only if for every $x \in \K$ and $n \in N_\K(x)$ we have
\[
  \duality{n}{z-x} \le \frac{\norm{n}{}}{2r}\norm{z-x}{}^2, \qquad \forall z \in \K.
\]
\end{Thm}

\subsection{Functions of bounded variation}

Let $I$ be an interval of $\re$. The set of $\H$-valued continuous functions defined on $I$ is denoted by $\Czero(I;\H)$. For a function $f : I \function \H$ and for $S \subseteq I$ we write 
$\Lipcost(f,S) := \sup\{\d(f(s),f(t))/|t-s|\ :\ s, t \in S,\ s \neq t\}$, $\Lipcost(f) := \Lipcost(f,I)$, the Lipschitz constant of $f$, and $\Lip(I;\X) := \{f : I \function \H\ :\ \Lipcost(f) < \infty\}$, the set of 
$\H$-valued Lipschitz continuous functions on $I$. 

\begin{Def}
Given an interval $I \subseteq \re$, a function $f : I \function \H$, and a subinterval $J \subseteq I$, the \emph{variation of $f$ on $J$} is defined by
\begin{equation}\notag
  \pV(f,J) := 
  \sup\left\{
           \sum_{j=1}^{m} \d(f(t_{j-1}),f(t_{j}))\ :\ m \in \en,\ t_{j} \in J\ \forall j,\ t_{0} < \cdots < t_{m} 
         \right\}.
\end{equation}
If $\pV(f,I) < \infty$ we say that \emph{$f$ is of bounded variation on $I$} and we set 
\[
  \BV(I;\H) := \{f \in I : \function \H\ :\ \pV(f,I) < \infty\}.
\]
\end{Def}

It is well known that the completeness of $\H$ implies that every $f \in \BV(I;\H)$ admits one sided limits $f(t-), f(t+)$ at every point $t \in I$, with the convention that $f(\inf I-) := f(\inf I)$ if $\inf I \in I$, and that $f(\sup I+) := f(\sup I)$ if $\sup I \in I$. If $I$ is bounded we have 
$\Lip(I;\H) \subseteq \BV(I;\H)$.

\subsection{Differential measures}\label{differential measures}

Given an interval $I$ of the real line $\mathbb{R}$, the family of Borel sets in $I$ is denoted by 
$\borel(I)$. If $\mu : \borel(I) \function \clint{0,\infty}$ is a measure, $p \in \clint{1,\infty}$, then the space of $\H$-valued functions which are $p$-integrable with respect to $\mu$ will be denoted by 
$\L^p(I, \mu; \H)$ or simply by $\L^p(\mu; \H)$. For the theory of integration of vector valued functions we refer, e.g., to \cite[Chapter VI]{Lan93}. 
When $\mu = \leb^1$, where $ \leb^1$ is  the one dimensional Lebesgue measure, 
we write  $\L^p(I; \H) := \L^p(I,\mu; \H)$.

We recall that a \emph{$\H$-valued measure on $I$} is a map $\nu : \borel(I) \function \H$ such that 
$\nu(\bigcup_{n=1}^{\infty} B_{n})$ $=$ $\sum_{n = 1}^{\infty} \nu(B_{n})$ for every sequence 
$(B_{n})$ of mutually disjoint sets in $\borel(I)$. The \emph{total variation of $\nu$} is the positive measure $\vartot{\nu} : \borel(I) \function \clint{0,\infty}$ defined by
\begin{align}\label{tot var measure}
  \vartot{\nu}(B)
  := \sup\left\{\sum_{n = 1}^{\infty} \norm{\nu(B_{n})}{}\ :\ 
                 B = \bigcup_{n=1}^{\infty} B_{n},\ B_{n} \in \borel(I),\ 
                 B_{h} \cap B_{k} = \varnothing \text{ if } h \neq k\right\}. \notag
\end{align}
The vector measure $\nu$ is said to be \emph{with bounded variation} if $\vartot{\nu}(I) < \infty$. In this case the equality $\norm{\nu}{} := \vartot{\nu}(I)$ defines a complete norm on the space of measures with bounded variation (see, e.g. \cite[Chapter I, Section  3]{Din67}). 

If $\mu : \borel(I) \function \clint{0,\infty}$ is a positive bounded Borel measure and if 
$g \in \L^1(I,\mu;\H)$, then $g\mu : \borel(I) \function \H$ denotes the vector measure defined by 
\begin{equation}
  g\mu(B) := \int_B g\de \mu, \qquad B \in \borel(I). \notag
\end{equation} 
In this case we have that 
\[
  \vartot{g\mu}(B) = \int_B \norm{g(t)}{}\de \mu \qquad \forall B \in \borel(I)
\] 
(see \cite[Proposition 10, p. 174]{Din67}). 

Assume that $\nu : \borel(I) \function \H$ is a vector measure with bounded variation and $f : I \function \H$ and $\phi : I \function \mathbb{R}$ are two \emph{step maps with respect to $\nu$}, i.e. there exist $f_{1}, \ldots, f_{m} \in \H$, 
$\phi_{1}, \ldots, \phi_{m} \in \H$ and $A_{1}, \ldots, A_{m} \in \borel(I)$ mutually disjoint such that 
$\vartot{\nu}(A_{j}) < \infty$ for every $j$ and $f = \sum_{j=1}^{m} \indicator_{A_{j}} f_{j}$, 
$\phi = \sum_{j=1}^{m} \indicator_{A_{j}} \phi_{j},$ where $\indicator_{S} $ is the characteristic function of a set $S$, i.e. 
$\indicator_{S}(x) := 1$ if $x \in S$ and $\indicator_{S}(x) := 0$ if $x \not\in S$. For such step maps we define 
$\int_{I} \duality{f}{\de\nu} := \sum_{j=1}^{m} \duality{f_{j}}{\nu(A_{j})} \in \mathbb{R}$ and
$\int_{I} \phi \de \nu := \sum_{j=1}^{m} \phi_{j} \nu(A_{j}) \in \H$.

 If $\Step(\vartot{\nu};\H)$ (resp. $\Step(\vartot{\nu})$) is the set of $\H$-valued (resp. real valued) step maps with respect to $\nu$, then the maps
$\Step(\vartot{\nu};\H)$ $\function$ $\H : f \longmapsto \int_{I} \duality{f}{\de\nu}$ and
$\Step(\vartot{\nu})$ $\function$ $\H : \phi \longmapsto \int_{I} \phi \de \nu$ 
are linear and continuous when $\Step(\vartot{\nu};\H)$ and $\Step(\vartot{\nu})$ are endowed with the 
$\L^{1}$-seminorms $\norm{f}{\L^{1}(\vartot{\nu};\H)} := \int_I \norm{f}{} \de \vartot{\nu}$ and
$\norm{\phi}{\L^{1}(\vartot{\nu})} := \int_I |\phi| \de \vartot{\nu}$. Therefore they admit unique continuous extensions 
$\mathsf{I}_{\nu} : \L^{1}(\vartot{\nu};\H) \function \mathbb{R}$ and 
$\mathsf{J}_{\nu} : \L^{1}(\vartot{\nu}) \function \H$,
and we set 
\[
  \int_{I} \duality{f}{\de \nu} := \mathsf{I}_{\nu}(f), \quad
  \int_{I} \phi\, \de\nu := \mathsf{J}_{\nu}(\phi),
  \qquad f \in \L^{1}(\vartot{\nu};\H),\quad \phi \in \L^{1}(\vartot{\nu}).
\]

If $\mu$ is a bounded positive measure and $g \in \L^{1}(\mu;\H)$, arguing first on step functions, and then taking limits, it is easy to check that 
\[
  \int_I\duality{f}{\de(g\mu)} = \int_I \duality{f}{g}\de \mu, \qquad \forall f \in \L^{\infty}(\mu;\H).
\] 
The following results (cf., e.g., \cite[Section III.17.2-3, p. 358-362]{Din67}) provide a connection between functions with bounded variation and vector measures which will be implicitly used in this paper.

\begin{Thm}\label{existence of Stietjes measure}
For every $f \in \BV(I;\H)$ there exists a unique vector measure of bounded variation $\nu_{f} : \borel(I) \function \H$ such that 
\begin{align}
  \nu_{f}(\opint{c,d}) = f(d-) - f(c+), \qquad \nu_{f}(\clint{c,d}) = f(d+) - f(c-), \notag \\ 
  \nu_{f}(\clsxint{c,d}) = f(d-) - f(c-), \qquad \nu_{f}(\cldxint{c,d}) = f(d+) - f(c+). \notag 
\end{align}
whenever $\inf I \le c < d \le \sup I$ and the left hand side of each equality makes sense.

\noindent Conversely, if $\nu : \borel(I) \function \H$ is a vector measure with bounded variation, and if $f_{\nu} : I \function \H$ is defined by 
$f_{\nu}(t) := \nu(\clsxint{\inf I,t} \cap I)$, then $f_{\nu} \in \BV(I;\H)$ and $\nu_{f_{\nu}} = \nu$.
\end{Thm}

\begin{Prop}
Let $f  \in \BV(I;\H)$, let $g : I \function \H$ be defined by $g(t) := f(t-)$, for $t \in \Int(I)$, and by $g(t) := f(t)$, if $t \in \partial I$, and let $V_{g} : I \function \mathbb{R}$ be defined by 
$V_{g}(t) := \pV(g, \clint{\inf I,t} \cap I)$. Then  
$\nu_{g} = \nu_{f}$ and $\vartot{\nu_{f}} = \nu_{V_{g}} = \pV(g,I)$.
\end{Prop}

The measure $\nu_{f}$ is called the \emph{Lebesgue-Stieltjes measure} or \emph{differential measure} of $f$. Let us see the connection of the differential measure with the distributional derivative. If $f \in \BV(I;\H)$ and if $\overline{f}: \mathbb{R} \function \H$ is defined by
\begin{equation}\label{extension to R}
  \overline{f}(t) :=
  \begin{cases}
    f(t) 	& \text{if $t \in I$}, \\
    f(\inf I)	& \text{if $\inf I \in \mathbb{R}$, $t \not\in I$, $t \le \inf I$},\\
    f(\sup I)	& \text{if $\sup I \in \mathbb{R}$, $t \not\in I$, $t \ge \sup I,$}
  \end{cases}
\end{equation}
then, as in the scalar case, it turns out (cf. \cite[Section 2]{Rec11}) that $\nu_{f}(B) = \De \overline{f}(B)$ for every 
$B \in \borel(\mathbb{R})$, where $\De\overline{f}$ is the distributional derivative of $\overline{f}$, i.e.
\[
  - \int_\mathbb{R} \varphi'(t) \overline{f}(t) \de t = \int_{\mathbb{R}} \varphi \de \De \overline{f}, 
  \qquad \forall \varphi \in \Czero_{c}^{1}(\mathbb{R};\mathbb{R}),
\]
where $\Czero_{c}^{1}(\mathbb{R};\mathbb{R})$ is the space of continuously differentiable functions on $\mathbb{R}$ with compact support.
Observe that $\De \overline{f}$ is concentrated on $I$: $\De \overline{f}(B) = \nu_f(B \cap I)$ for every $B \in \borel(I)$, hence in the remainder of the paper, if $f \in \BV(I,\H)$ then we will simply write
\begin{equation}
  \De f := \De\overline{f} = \mu_f, \qquad f \in \BV(I;\H),
\end{equation}
and from the previous discussion it follows that 
\begin{equation}\label{D-TV-pV}
  \norm{\De f}{} = \vartot{\De f}(I) = \norm{\nu_f}{}  = \pV(f,I), \qquad \forall f \in \BV(I;\H).
\end{equation}
If $I$ is bounded and $p \in \clint{1,\infty}$, then the classical Sobolev space $\W^{1,p}(I;\H)$ consists of those functions $f \in \Czero(I;\H)$ for which $\De f = g\leb^1$ for some $g \in \L^p(I;\H)$ and  $\W^{1,\infty}(I;\H) = \Lip(I;\H)$. Let us also recall that if $f \in \W^{1,1}(I;\H)$ then the derivative $f'(t)$ exists $\leb^1$-a.e. in $t \in I$, 
$\De f = f' \leb^1$, and $\V(f,I) = \int_I\norm{f'(t)}{}\de t$ (cf., e.g. \cite[Appendix]{Bre73}).


\section{Main results}\label{S:state main result}

From now on we will assume that
\begin{equation}\label{Z}
  \text{$\Z$ is a $r$-prox-regular subset of $\H$ for some $r > 0$}, 
\end{equation}
and
\begin{equation}\label{T}
 T > 0.
\end{equation}
We will consider on $\BV(\clint{0,T};\H)$ the classical complete $\BV$-norm defined by
\eqref{def BVnorm},
where 
\[
  \norm{f}{\infty} := \sup\{\norm{f(t)}{}\ :\ t \in \clint{0,T}\}.
\]
The norm \eqref{def BVnorm} is equivalent to the norm defined by 
\[
  \interleave f \interleave_{\BV} := \norm{f(0)}{} + \V(f,\clint{0,T}), \qquad f \in \BV(\clint{0,T};\H).
\]
From \eqref{D-TV-pV} it also follows that
\[
\norm{f}{\BV} = \norm{f}{\infty} + \norm{\De f}{} = \norm{f}{\infty} + \vartot{\De f}(\clint{0,T}),\qquad 
\forall f \in \BV(\clint{0,T};\H).
\]
where $\De f$ is the differential measure of $f$ and $\vartot{\De f}$ is the total variation measure of $\De f$. 

\noindent We also have
\[
   \norm{f}{\BV} = \norm{f}{\infty} + \int_0^T\norm{f'(t)}{}\de t \qquad \forall f \in \W^{1,1}(\clint{0,T};\H).
\]
On $\BV(\clint{0,T};\H)$ we will consider also the so-called \emph{strict metric} defined by \eqref{def strictBV}.
We say that \emph{$f_n \to f$ strictly on $\clint{0,T}$} if $\d_{s}(f_n,f) \to 0$ as $n \to \infty$. Let us recall that $\d_{s}$ is not complete and the topology induced by $\d_{s}$ is not linear.

We now state the main problem of our paper.
The solution operator of this problem is classically called ``play operator'' in the case when the characteristic set $\Z$ of constraint is convex. We will call it ``play operator'' also in the non-convex case, or we will use the term ``non-convex play operator'' in order to emphasize the non-convexity of the set of constraints.

\begin{Pb}\label{CBVplay}
Assume that \eqref{H-prel}, \eqref{Z}, \eqref{T} hold. For any 
$u \in \Czero(\clint{0,T};\H) \cap \BV(\clint{0,T};\H)$ and any $z_0 \in \Z$ one has to find 
$y = \Pl(u,z_0) \in \Czero(\clint{0,T};\H) \cap \BV(\clint{0,T};\H)$ such that 
\begin{align}
  & u(t) - y(t) \in \Z \quad \forall t \in \clint{0,T}, \label{CBV x in Z}\\
  & \int_{\clint{0,T}} \duality{z(t)-u(t)+y(t)}{\de\De y(t)} \le 
      \frac{1}{2r}\int_{\clint{0,T}} \norm{z(t)-u(t)+y(t)}{}^2 \de \vartot{\De y}(t) \notag\\
  & \hspace{45ex}\forall z \in \BV(\clint{0,T};\H), z(\clint{0,T}) \subseteq \Z, \label{CBV v.i.} \\
  & u(0) - y(0) = z_0. \label{i.c.}
\end{align}
\end{Pb}

The integrals in \eqref{CBV v.i.} are Lebesgue integrals with respect to the measures $\De y$
and $\vartot{\De y}$. The inequality can be equivalently written using Riemann-Stieltjes integrals, by virtue of \cite[Lemma A.9]{Rec11} and the discussion in Section \ref{differential measures}.

Problem \ref{CBVplay} can be equivalently stated as a differential inclusion. Indeed we have the following:

\begin{Prop}\label{int form}
Assume that \eqref{H-prel}, \eqref{Z}, \eqref{T} hold and that 
$u \in \Czero(\clint{0,T};\H) \cap \BV(\clint{0,T};\H)$ and $z_0 \in \Z$. Then a function 
$y \in \Czero(\clint{0,T};\H) \cap \BV(\clint{0,T};\H)$ is a solution to Problem \ref{CBVplay} if and only if there exists a measure $\mu : \borel(\clint{0,T}) \function \clsxint{0,\infty}$ and a function 
$v \in \L^1(\mu,\H)$ such that
\begin{align}
  & \De y = v \De \mu, \label{Dy=vmu}\\
  & u(t) - y(t) \in \Z \qquad \forall t \in \clint{0,T} \\
  & -v(t) \in N_{u(t)-\Z}(y(t)) \qquad \text{for $\mu$-a.e. $t \in \clint{0,T}$} \\
  & u(0) - y(0) = z_0. \label{i.c.2}
\end{align}
\end{Prop}

\begin{proof}
First of all let us show that \eqref{CBV v.i.} can be equivalently written either with 
$z \in \BV(\clint{0,T};\H)$ or with a right-continuous $z \in \Reg(\clint{0,T};\H)$, the space of 
$\H$-valued regulated functions on $\clint{0,T}$, i.e. those functions $z : \clint{0,T} \function \H$ for which there exists one-sided limits $z(t-)$ and $z(t+)$ for every $t \in \clint{0,T}$, with the convention that $z(0-) = z(0)$ and $z(T+) = z(T)$. Indeed if $y \in \Czero(\clint{0,T};\H) \cap \BV(\clint{0,T};\H)$ satisfies \eqref{CBV v.i.} and if $z \in \Reg(\clint{0,T};\H)$, then by 
\cite[Theorem 3, Section 2.1]{Bou58} there exists a sequence of step functions 
$z_n \in \BV(\clint{0,T};\H)$ such that $z_n \to z$ uniformly on $\clint{0,T}$. Hence taking the limit in \eqref{CBV v.i.} with $z$ replaced by $z_n$, we get that \eqref{CBV v.i.} holds with any 
right-continuous $z \in \Reg(\clint{0,T};\H)$ such that $z(\clint{0,T}) \subseteq \Z$. Now we can conclude by recalling that by \cite[Theorem 3.7]{KreMonRec22b} we have that 
\eqref{CBV x in Z}--\eqref{i.c.} is equivalent to the existence of a positive measure $\mu$ and of a function $v \in \L^1(\mu;\clint{0,T})$ such that \eqref{Dy=vmu}--\eqref{i.c.2} hold.
\end{proof}

The first proof of the existence of a solution to the non-convex Problem \ref{CBVplay} can be found in \cite{EdmThi06} (but see also \cite{CasMon96,ColGon99, Ben00, BouThi05}). To be more precise, in \cite[Corollary 3.1]{EdmThi06} it is proved that there exists a unique solution to 
\eqref{Dy=vmu}--\eqref{i.c.2}. Thus by virtue of Proposition \ref{int form} we have the following:

\begin{Thm}\label{CBVExist}
Assume that \eqref{H-prel}, \eqref{Z}, \eqref{T} hold. Then Problem \ref{CBVplay} has a unique solution for any 
$u \in \Czero(\clint{0,T};\H) \cap \BV(\clint{0,T};\H)$ and any $z_0 \in \Z$.
\end{Thm}

Another proof of Theorem \ref{CBVExist} can be found in \cite{KreMonRec22a}, exclusively within the framework of the integral formulation.

\begin{Def}
The solution operator 
$\Pl : \Czero(\clint{0,T};\H) \cap \BV(\clint{0,T};\H) \times \Z \function 
\Czero(\clint{0,T};\H) \cap \BV(\clint{0,T};\H)$ associating to every 
$(u,z_0) \in \Czero(\clint{0,T};\H) \cap \BV(\clint{0,T};\H) \times \Z$ the unique solution 
$y = \Pl(u,z_0)$ of Problem \ref{CBVplay}, is called the \emph{(non-convex) play operator}.
\end{Def}

When the ``input'' function $u$ of Problem \ref{CBVplay} is more regular, we have the following 
well-known characterization of solutions (see, e.g., \cite[Corollary 6.3]{KreMonRec22a}).

\begin{Prop}
Assume that \eqref{H-prel}, \eqref{Z}, \eqref{T} hold. If $u \in \W^{1,p}(\clint{0,T};\H)$, $z_0 \in \Z$,
and if $y = \Pl(u,z_0)$ is the solution of Problem \ref{CBVplay}, then $y \in \W^{1,p}(\clint{0,T};\H)$
and
\begin{align}
  & u(t) - y(t) \in \Z \quad \forall t \in \clint{0,T}, \\
  & \duality{z-u(t)+y(t)}{y'(t)} \le \frac{\norm{y'(t)}{}}{2r} \norm{z(t)-u(t)+y(t)}{}^2 \quad 
      \text{for $\leb^1$-a.e. $t \in \clint{0,T}$, $\forall z \in \Z$,}  \label{W1p v.i.} \\
  & u(0) - y(0) = z_0. \label{W1p i.c.}
\end{align}
Moreover $y$ is the unique function in $\W^{1,p}(\clint{0,T};\H)$ such that 
\eqref{W1p v.i.}--\eqref{W1p i.c.} holds.
\end{Prop} 

Now we can state our main theorems. The first result states that $\Pl$ is continuous with respect to the $\BV$-norm on $\Czero(\clint{0,T};\H) \cap \BV(\clint{0,T};\H)$.

\begin{Thm}\label{T:BVnorm cont}
Assume that \eqref{H-prel}, \eqref{Z}, \eqref{T} hold. The play operator 
\newline $\Pl : \Czero(\clint{0,T};\H) \cap \BV(\clint{0,T};\H) \times \Z \function 
\Czero(\clint{0,T};\H) \cap \BV(\clint{0,T};\H)$ is continuous with respect to the $\BV$-norm \eqref{def BVnorm}, i.e. if 
\[
  \norm{u-u_n}{\BV} \to 0, \quad \norm{z_0 - z_{0n}}{} \to 0 \qquad \text{as $n \to \infty$},
\]
then
\[
 \norm{\Pl(u,z_0) - \Pl(u_n,z_{0n})}{\BV} \to 0 \qquad \text{as $n \to \infty$}
\]
whenever $u, u_n \in \Czero(\clint{0,T};\H) \cap \BV(\clint{0,T};\H)$ and $z_0, z_{0,n} \in \Z$ for every $n \in \en$.
\end{Thm}

We will also prove that the play operator is continuous with respect to the strict metric.

\begin{Thm}\label{T:BVstrict cont}
Assume that \eqref{H-prel}, \eqref{Z}, \eqref{T} hold. The play operator 
$\Pl : \Czero(\clint{0,T};\H) \cap \BV(\clint{0,T};\H) \times \Z \function 
\Czero(\clint{0,T};\H) \cap \BV(\clint{0,T};\H)$ is continuous with respect to the strict metric $\d_s$ \eqref{def strictBV}, i.e. if 
\[
  \d_s(u,u_n)\to 0, \quad \norm{z_0 - z_{0n}}{} \to 0 \qquad \text{as $n \to \infty$},
\]
then
\[
 \d_s(\Pl(u,z_0), \Pl(u_n,z_{0n})) \to 0 \qquad \text{as $n \to \infty$}
\]
whenever $u, u_n \in \Czero(\clint{0,T};\H) \cap \BV(\clint{0,T};\H)$ and $z_0, z_{0,n} \in \Z$ for every $n \in \en$.
\end{Thm}

The proofs of our main theorems are strongly based on the fact that the play operator is rate independent, which is the property \eqref{P r.i.} of $\Pl$ proved in the following theorem.

\begin{Thm}\label{th:P r.i.}
Assume that \eqref{H-prel}, \eqref{Z}, \eqref{T} hold, $u \in \Czero(\clint{0,T};\H) \cap \BV(\clint{0,T};\H)$, and $z_0 \in \Z$. If $\phi : \clint{0,T} \function \clint{0,T}$ is a continuous function such that 
$(\phi(t) - \phi(s))(t-s) \ge 0$ and $\phi(\clint{0,T}) = \clint{0,T}$ and $ \Pl(u,z_0)$ is the solution of Problem \ref{CBVplay}, then 
\begin{equation}\label{P r.i.}
  \Pl(u \circ \phi,z_0) = \Pl(u,z_0) \circ \phi.
\end{equation}
\end{Thm}

We will prove Theorems \ref{T:BVnorm cont}, \ref{T:BVstrict cont}, and \ref{th:P r.i.} in Section \ref{S:proofs}.


\section{Proofs}\label{S:proofs}

Let us start by proving that $\Pl$ is rate independent.

\begin{proof}[Proof of Theorem \ref{th:P r.i.}]

Set $y := \Pl(u,z_0)$, and recall that $V_y(t) = \V(y,\clint{0,t})$ for every $t \in \clint{0,T}$. Hence $\vartot{\De y} = \De V_y$, and by the vectorial Radon-Nikodym theorem 
(\cite[Corollary VII.4.2]{Lan93}) there exists $v \in \L^1(\vartot{\De y};\H)$ such that 
$\De y = v \De V_y$. Let us fix $z \in \BV(\clint{0,T};\H)$ such that 
$z(\clint{0,T}) \subseteq \clint{0,T}$ and recall the following well-known formula holding for any 
measure $\mu : \borel(\clint{0,T}) \function \clsxint{0,\infty}$, $g \in \L^1(\mu;\H)$, and 
$A \in \borel(\clint{0,T})$:
\[
  \int_{\phi^{-1}(A)} g(\phi(t)) \de\mu(t) = \int_{A} g(\tau) \de(\phi_* \mu)(\tau),
\]
where $\phi_* \mu : \borel(\clint{0,T}) \function \clsxint{0,\infty}$ is the measure defined by 
$\phi_* \mu(B) := \mu(\phi^{-1}(B))$ for $B \in \borel(\clint{0,T})$ (this formula can be proved by approximating $g$ by a sequence of step functions and then taking the limit).
If $0 \le \alpha \le \beta \le T$ we have 
\[
\phi_*(\De V_y\circ \phi)(\clint{\alpha,\beta}) = 
(\De V_y \circ \phi)(\phi^{-1}(\clint{\alpha,\beta})) =\De V_y(\clint{\alpha,\beta}),
\]
hence 
\[
\phi_*(\De V_y\circ \phi) = \De V_y,
\] 
and for $0 \le a \le b \le T$ we find
\begin{align}
  \De\ \!(y\circ \phi)(\clint{a,b}) 
  & = y(\phi(b)) - y(\phi(a)) = \De y(\clint{\phi(a),\phi(b)})\notag \\
  & = \int_{\clint{\phi(a),\phi(b)}} v(\tau)\de \De V_y(\tau) = 
   \int_{\clint{\phi(a),\phi(b)}} v(\tau)\de \De\ \!(\phi_* V_y)(\tau) \notag \\
  & = \int_{\clint{a,b}} v(\phi(t))\de \De\ \!(V_y \circ \phi)(t) =
  (v \circ \phi)\De\ \!(V_y \circ \phi)(\clint{a,b}), \notag
\end{align}
so that
\[
  \De\ \!(y\circ \phi) = (v \circ \phi)\De\ \!(V_y \circ \phi), \qquad
  \vartot{\De\ \!(y\circ \phi)} = \norm{v \circ \phi}{} \De\ \!(V_y \circ \phi).
\]
If $\psi(\tau) := \inf \phi^{-1}(\tau)$, then $\psi$ is increasing and $\tau = \phi(\psi(\tau))$. Therefore, since $\De\ \!(V_y\circ \phi)=0$ on every interval where $\phi$ is constant, we find that for every 
$h \in \Czero(\re^2)$ we have 
\[
\int_{\clint{0,T}} h(z(t),\phi(t))\de\De\ \!(V_y\circ \phi)(t) =
\int_{\clint{0,T}} h(z(\psi(\phi(t)), \phi(t)) \de\De\ \!(V_y\circ \phi)(t).
\]
Hence
\begin{align}
  &\int_{\clint{0,T}} \duality{z(t)-u(\phi(t))+y(\phi(t))}{\de\De\ \!(y\circ \phi)(t)} \notag\\
  &=  \int_{\clint{0,T}} \duality{z(t)-u(\phi(t))+y(\phi(t))}{v(\phi(t))}\de\De\ \!(V_y\circ \phi)(t) \notag \\
  &=  \int_{\clint{0,T}} \duality{z(\psi(\phi(t)))-u(\phi(t))+y(\phi(t))}{v(\phi(t))}\de\De\ \!(V_y\circ \phi)(t) \notag \\
  &=  \int_{\clint{0,T}} \duality{z(\psi(\tau))-u(\tau)+y(\tau)}{v(\tau)}{\de\De V_y(\tau)} \notag \\
  & = \int_{\clint{0,T}} \duality{z(\psi(\tau))-u(\tau)+y(\tau)}{\de\De y(\tau)}, \label{r-i1}
\end{align}
and
\begin{align}
& \int_{\clint{0,T}} \norm{z(t)-u(\phi(t))+y(\phi(t))}{}^2 \de \vartot{\De\ \!(y\circ \phi)}(t) \notag \\
& = \int_{\clint{0,T}} \norm{z(t)-u(\phi(t))+y(\phi(t))}{}^2 v(\phi(t)) \de \De\ \!(V_y\circ \phi)(t) \notag \\
& = \int_{\clint{0,T}} \norm{z(\psi(\phi(t))-u(\phi(t))+y(\phi(t))}{}^2 \norm{v(\phi(t))}{} 
         \de \De\ \!(V_y\circ \phi)(t) \notag \\
& = \int_{\clint{0,T}} \norm{z(\psi(\tau))-u(\tau)+y(\tau)}{}^2 \norm{v(\tau)}{} \de \De V_y(\tau) \notag \\
& = \int_{\clint{0,T}} \norm{z(\psi(\tau))-u(\tau)+y(\tau)}{}^2  \de \vartot{\De y}(\tau) \label{r-i2}.
\end{align}
Since $y = \Pl(u,z_0)$  we have that the right hand  side of \eqref{r-i1} is less  or equal to the right hand side of \eqref{r-i2} times $\frac{1}{2r}$ and this implies that 
\begin{align}
& \int_{\clint{0,T}} \duality{z(t)-u(\phi(t))+y(\phi(t))}{\de\De\ \!(y\circ \phi)(t)} \notag \\
& \le
\frac{1}{2r}\int_{\clint{0,T}} \norm{z(t)-u(\phi(t))+y(\phi(t))}{}^2 \de \vartot{\De\ \!(y\circ \phi)}(t),
\end{align}
which is what we wanted to prove.
\end{proof}

In the next result, we prove a normality rule for the non-convex play operator, thereby  we generalize
to the non-convex case the result in \cite[Proposition 3.9]{Kre96}. The idea of the proof is analogous to the one of \cite[Proposition 3.9]{Kre96}.

\begin{Prop}\label{S and Q}
Assume that \eqref{H-prel}, \eqref{Z}, \eqref{T} hold, $u \in \Lip(\clint{0,T};\H)$, $z_0 \in \Z$, and that $y = \Pl(u,z_0)$. Let $x = \Stop(u, z_0) : \clint{0,T} \function \H$ and 
$w = \Q(u, z_0) : \clint{0,T} \function \H$ be defined by
\begin{align}
  x(t) := \Stop(u, z_0)(t) := u(t) - y(t), \qquad \text{$t \in \clint{0,T}$,} \label{def S}\\
  w(t) := \Q(u, z_0)(t) := y(t) - x(t), \qquad \text{$t \in \clint{0,T}$.} \label{def Q}
\end{align}
Then $w = \Q(u, z_0) \in \Lip(\clint{0,T};\H)$, $x = \Stop(u, z_0) \in \Lip(\clint{0,T};\H)$, $x(t) \in Z$ for every $t \in \clint{0,T}$, and 
\begin{equation}\label{y'.x'=0}
  \duality{y'(t)}{x'(t)} = 0 \qquad \text{for $\leb^1$-a.e. $t \in \clint{0,T}$,}
\end{equation}
and
\begin{equation}\label{|w'|=|u'|}
  \norm{w'(t)}{} = \norm{u'(t)}{} \qquad \text{for $\leb^1$-a.e. $t \in \clint{0,T}$.}
\end{equation}
\end{Prop}

\begin{proof}
Let $t \in \clint{0,T}$ be a point where $x$ is differentiable. 
Taking $z(t) = x(t+h) \in \Z$ for every $h \in \re$ sufficiently small, we have
\[
  \frac{1}{h}\duality{y'(t)}{x(t) - x(t+h)} \ge -\frac{\norm{y'(t)}{}}{2rh}\norm{x(t) - x(t+h)}{}^2
\]
therefore letting $h \to 0$ we get
\begin{equation}\label{y'.x'<0}
  \duality{y'(t)}{-x'(t)} \ge 0.
\end{equation}
Taking $z(t) = x(t-h)$ we also have
\[
  \frac{1}{h}\duality{y'(t)}{x(t) - x(t-h)} \ge -\frac{\norm{y'(t)}{}}{2rh}\norm{x(t) - x(t-h)}{}^2
\]
therefore letting $h \to 0$ we get
\[
  \duality{y'(t)}{x'(t)} \ge 0,
\]
which together with \eqref{y'.x'<0} yields \eqref{y'.x'=0}. This formula implies that 
\begin{equation}\label{|w'|}
  \norm{w'(t)}{}^2 = \norm{y'(t) - x'(t)}{}^2 = \duality{y'(t) - x'(t)}{y'(t) - x'(t)}
  = \norm{y'(t)}{}^2 + \norm{x'(t)}{}^2,
\end{equation}
and
\begin{equation}\label{|u'|}
  \norm{u'(t)}{}^2 = \norm{y'(t) + x'(t)}{}^2 = \duality{y'(t) + x'(t)}{y'(t) + x'(t)}
  = \norm{y'(t)}{}^2 + \norm{x'(t)}{}^2,
\end{equation}
therefore \eqref{|w'|=|u'|} follows.
\end{proof}

Let us observe that, in the previous proposition, the geometrical meaning of \eqref{|w'|}-\eqref{|u'|}, is that $w'(t)$ and $u'(t)$ are the diagonals of the rectangle with sides $x'(t)$ and $y'(t)$, so that we have \eqref{|w'|=|u'|}.

Now we prove the continuity of the play operator with respect to the $\BV$-norm on the space
$\Lip(\clint{0,T};\H)$. Our proof is based on the normality rule of Proposition \ref{S and Q} and on a very simple application of a standard weak convergence-argument in $\L^2(\clint{0,T};\H)$.

\begin{Thm}
Assume that \eqref{H-prel}, \eqref{Z}, \eqref{T} hold. The non-convex play operator restricted to $\Z \times \Lip(\clint{0,T};\H)$, i.e. $\Pl : \Z \times \Lip(\clint{0,T};\H) \function \Lip(\clint{0,T};\H)$, is continuous with respect to the $\BV$-norm. More precisely let $z_0 \in \Z$, 
$u \in \Lip(\clint{0,T};\H)$, $z_{0,n} \in \Z$, $u_n \in \Lip(\clint{0,T};\H)$ for every 
$n \in \en$, and let $y := \Pl(z_0, u)$ and $y_n := \Pl(z_{0,n}, u_n).$ 

\noindent If $\norm{z- z_n}{} \to 0$ and
\[
  \norm{u - u_n}{\infty} + \norm{u' - u_n'}{\L^1(\clint{0,T};\H)} \to 0 
  \qquad \text{as $n \to \infty$},
\]
then
\[
  \norm{y - y_n}{\infty} + \norm{y' - y_n'}{\L^1(\clint{0,T};\H)} \to 0 
  \qquad \text{as $n \to \infty$}.
\]
\end{Thm}

\begin{proof}
For every $n \in \en$ let us set $w := \Q(u,z_0)$ and $w_n := \Q(u_n,z_0)$ according to formulas
\eqref{def S}--\eqref{def Q}. From Proposition \ref{S and Q} we find that
\begin{equation}
  \norm{w_n'}{\L^2(\clint{0,T};\H)} = \norm{u_n'}{\L^2(\clint{0,T};\H)} \to 
  \norm{u'}{\L^2(\clint{0,T};\H)} = \norm{w'}{\L^2(\clint{0,T};\H)} \qquad \text{as $n \to \infty$.}
\end{equation}
In particular $\{w_n'\}$ is bounded in $\L^2(\clint{0,T};\H)$, therefore there exists 
$\eta \in \L^2(\clint{0,T};\H)$ such that, at least for a subsequence, 
\begin{equation}
  w_n' \convergedeb \eta \qquad \text{in $\L^2(\clint{0,T};\H)$}.
\end{equation}
But $y_n \to y$ uniformly on $\clint{0,T}$ (see the proof of Theorem 5.5 in \cite{KreMonRec22a}), hence $w_n \to w$ uniformly on $\clint{0,T}$, therefore
\begin{equation}
  w_n' \convergedeb  w' \qquad \text{in $\L^2(\clint{0,T};\H)$}.
\end{equation}
Therefore since $\L^2(\clint{0,T};\H)$ is a Hilbert space we infer that
\begin{equation}
  w_n' \to  w' \qquad \text{in $\L^2(\clint{0,T};\H)$},
\end{equation}
which implies that
\begin{equation}
  w_n' \to  w' \qquad \text{in $\L^1(\clint{0,T};\H)$},
\end{equation}
so that $y_n' \to y'$ in $\L^1(\clint{0,T};\H)$ and we are done.
\end{proof}

Our proof of the $\BV$-norm continuity of $\Pl$ on $\Czero(\clint{0,T};\H) \cap \BV(\clint{0,T};\H)$
will essentially consists on reducing the problem to the Lipschitz continuous case by means of a reparametrization by the arc length. We need the following two auxiliary results. The first is the following:

\begin{Prop}\label{P:ftilde}
Assume that \eqref{H-prel} holds.
For every $f \in \Czero(\clint{0,T};\H) \cap \BV(\clint{0,T};\H)$, let $\ell_f : \clint{0,T} \function \re$ be defined by 
\begin{equation}\label{ell_f}
  \ell_f(t) = 
  \begin{cases}
    \dfrac{T}{\V(f, \clint{0,T})} \V(f,\clint{0,t}) & \text{if $\V(f, \clint{0,T}) \neq 0$}, \\
    \ \\
    0 & \text{if $\V(f, \clint{0,T}) = 0$},
  \end{cases}
\end{equation}
which we call \emph{normalized arc-length of $f$}. Then there exists 
$\ftilde \in \Lip(\clint{0,T};\H),$ the \emph{re\-pa\-ra\-me\-tri\-za\-tion of $f$ by the normalized arc-length}, such that 
\begin{equation}\label{f=ftilde(l)}
  f = \ftilde \circ \ell_f.
\end{equation}
Moreover
there exists a $\leb^1$-representative $\ftilde'$ of the distributional derivative of $\ftilde$ such that
\begin{equation}\label{|f'|=1}
  \norm{\ftilde'(\sigma)}{} = \frac{V(f,\clint{0,T})}{T}, \qquad \forall \sigma \in \clint{0,T}.
\end{equation}
\end{Prop}

\begin{proof}
The existence of a function $\ftilde \in \Lip(\clint{0,T};\H)$ satisfying \eqref{f=ftilde(l)} is easy to prove (see, e.g., \cite[Proposition 3.1]{Rec08}). Moreover we know from \cite[Lemma 4.3]{Rec11} that 
if $g$ is a $\leb^1$-representative of the distributional derivative, then 
$\norm{g(\sigma)}{} = V(f,\clint{0,T})/T$ for every $\sigma \in F$, for some $F \subseteq \clint{0,T}$
with 
full measure in $\clint{0,T}$.
Thus \eqref{|f'|=1} follows if we define the following Lebesgue representative 
of the derivative of $\ftilde$:
\[
  \ftilde'(\sigma) :=
  \begin{cases}
    g(\sigma) & \text{if $\sigma \in F$}, \\
    \ \\
    \dfrac{(V(f,\clint{0,T})}{T} e_0& \text{if $\sigma \not\in F$,}
  \end{cases}
\]
where $e_0 \in \H$ is chosen so that $\norm{e_0}{} = 1$.
\end{proof}

Then, as for the Lipschitz case, we need to introduce the operator $\Q$ defined by $\Q(v) = 2\Pl(v) - v$ for $v \in \Czero(\clint{0,T};\H) \cap \BV(\clint{0,T};\H)$.

\begin{Lem}\label{L:Q(v)}
Assume that $v \in\Czero(\clint{0,T};\H) \cap  \BV(\clint{0,T};\H)$, $z_0 \in \Z$, and let 
$\Q : \Czero(\clint{0,T};\H) \cap \BV(\clint{0,T};\H) \times \Z \function 
\Czero(\clint{0,T};\H) \cap \BV(\clint{0,T};\H)$ be defined by
\begin{equation}\label{Q in BV}
  \Q(v,z_0) := 2\Pl(v,z_0) - v, \qquad v \in \Czero(\clint{0,T};\H) \cap \BV(\clint{0,T};\H).
\end{equation}
Then $\Q$ is rate independent, i.e. 
\begin{equation}\label{Q r.i.}
  \Q(v \circ \phi ,z_0) = \Q(v,z_0) \circ \phi \qquad
   \forall v \in \Czero(\clint{0,T};\H) \cap \BV(\clint{0,T};\H)
\end{equation}
for every continuous function $\phi : \clint{0,T} \function \clint{0,T}$ such that 
$(\phi(t) - \phi(s))(t-s) \ge 0$ and $\phi(\clint{0,T}) = \clint{0,T}$. Moreover if $\ell_v$ is the arc-length
defined in \eqref{ell_f}, then
\begin{equation}\label{DQ = Q'Dl}
  \De \Q(v, z_0) = ((\Q(\vtilde, z_0))'\circ \ell_v)\De\ell_v,
\end{equation}
i.e.
\begin{equation}\label{DQ = Q'Dl-2}
  \De\Q(v, z_0)(B) = \int_B (\Q(\vtilde, z_0))'(\ell_v(t)) \de \De\ell_v(t), \qquad
  \forall B \in \borel(\clint{0,T}),
\end{equation}
where formulas \eqref{DQ = Q'Dl}--\eqref{DQ = Q'Dl-2} hold with any $\leb^1$-representative  
$(\Q(\vtilde, z_0))'$ of the distributional derivative of $\Q(\vtilde, z_0)$. Finally we can take such an
$\leb^1$-representative so that
\begin{equation}\label{|Q'|=1}
  \norm{(\Q(\vtilde, z_0))'(\sigma)}{} = \frac{\V(v,\clint{0,T})}{T},\qquad \forall \sigma \in \clint{0,T}.
\end{equation}
\end{Lem}

\begin{proof}
From Theorem \ref{th:P r.i.} it follows that
\[
  \Q(v \circ \phi,z_0) = 2\Pl(v\circ \phi,z_0) - v \circ \phi = 
  2\Pl(v,z_0) \circ \phi - v \circ \phi = \Q(v, z_0) \circ \phi,
\]
which is \eqref{Q r.i.}. Moreover, since $\vtilde$ is Lipschitz continuous, we have that 
$\Q(\vtilde, z_0) \in \Lip(\clint{0,T};\H)$, therefore by \cite[Theorem A.7]{Rec11} we infer that, if
$\Q(\vtilde, z_0)'$ is any $\leb^1$-representative of the distributional derivative of $\Q(\vtilde, z_0)$, then the bounded measurable function $\Q(\vtilde, z_0)'\circ \ell_v$ is a density of $\Q(v, z_0)$ with respect to the measure $\De\ell_v$, i.e. \eqref{DQ = Q'Dl} holds. Finally \eqref{|Q'|=1} follows from
\eqref{|w'|=|u'|} of Proposition \ref{S and Q} and from \eqref{|f'|=1} of Proposition \ref{P:ftilde}.
\end{proof}

Now we can prove our first main result.

\begin{proof}[Proof of Theorem \ref{T:BVnorm cont}]
Let us consider $u \in \BV(\clint{0,T};\H)$ and $u_n \in \BV(\clint{0,T};\H)$ for every $n \in \en$, and  assume that $\norm{u_n - u}{\BV(\clint{0,T};\H)} \to 0$ as $n \to \infty$. Then let $\ell := \ell_u$ and $\ell_n := \ell_{u_n}$ be the normalized arc-length functions defined in \eqref{ell_f}, so that we have
\[
  u = \utilde \circ \ell, \quad u_n = \utilde_n \circ \ell_n \qquad \forall n \in \en
\]
Let us also set
\begin{equation}\label{def w wn}
   w := \Q(u,z_0), \quad w_n := \Q(u_n, z_{0,n}), \qquad n \in \en,
\end{equation}
where the operator $\Q$ is defined in Lemma \ref{L:Q(v)}. By the proof of 
\cite[Theorem 5.5]{KreMonRec22a} we have that $\Pl(u_n,z_{0n}) \to \Pl(u_n,z_{0})$ uniformly
on $\clint{0,T}$, because $\norm{u_n -u}{\infty} \to 0$ as $n \to \infty$. Therefore from formula
\eqref{Q in BV} it follows that 
\begin{equation}\label{wn->w unif}
  w_n \to w \qquad \text{uniformly on $\clint{0,T}$}.
\end{equation}
Let us observe that $\Q(\utilde, z_0)$ and $\Q(\utilde_n, z_0)$ are Lipschitz continuous for every 
$n \in \en$ and let us define the bounded measurable functions $h: \clint{0,T} \function \H$ and $h_n: \clint{0,T} \function \H$ by
\begin{equation}
  h(t) := (\Q(\utilde, z_0))'(\ell_u(t)), \quad h_n(t) := (\Q(\utilde_n, z_{0n}))'(\ell_n(t)), 
  \qquad t \in \clint{0,T},
\end{equation}
where, by Lemma \ref{L:Q(v)}, formula \eqref{|Q'|=1}, we have that the $\leb^1$-representative 
of the distributional derivative of $\Q(\utilde, z_0)$ and $\Q(\utilde_n, z_0)$ can be chosen in such a way that
\[
 \norm{(\Q(\utilde, z_0))'(\sigma)}{} = \frac{\V(u,\clint{0,T})}{T}, \quad
\norm{(\Q(\utilde_n, z_0))'(\sigma)}{} = \frac{\V(u_n,\clint{0,T})}{T}, \qquad \forall\sigma \in \clint{0,T}.
\]
Therefore we have that
\begin{equation}\label{|gn(t)|, |g(t)|}
  \norm{h(t)}{} = \frac{\V(u,\clint{0,T})}{T}, \quad 
  \norm{h_n(t)}{} = \frac{\V(u_n,\clint{0,T})}{T},
  \qquad \forall t \in \clint{0,T}, \ \forall n \in \en.
\end{equation}
Since $u_n \to u$ in $\BV(\clint{0,T};\H)$, from the inequality 
\begin{equation}\label{V(f)-V(g)}
 |\V(u,\clint{a,b}) - \V(u_n,\clint{a,b})| \le \V(u-u_n,\clint{a,b}),
\end{equation}
holding for $0 \le a \le b \le T$, we infer that 
$\V(u_n,\clint{0,T}) \to \V(u,\clint{0,T})$ as $n \to \infty$, 
hence we have that the sequence 
$\{\V(u_n,\clint{0,T})\}$ is bounded. Therefore from \eqref{|gn(t)|, |g(t)|} we infer that there exists $C > 0$ such that
\begin{equation}\label{|g_n|infty bdd}
  \sup\{\norm{h_n(t)}{}\ :\ t \in \clint{0,T}\} \le C \qquad 
  \ \forall n \in \en
\end{equation}
and
\begin{equation}
  \lim_{n \to \infty} \norm{h_n(t)}{} = 
  \lim_{n \to \infty} \frac{\V(u_n,\clint{0,T})}{T} =
  \frac{\V(u,\clint{0,T})}{T} = \norm{h(t)}{}, \qquad \forall t \in \clint{0,T}.
\end{equation}
It follows that
\begin{align}
 \lim_{n \to \infty} \int_{\clint{0,T}} \norm{h_{n}(t)}{}^2\de\De\ell(t)
 & = \lim_{n \to \infty} \int_{\clint{0,T}} \left(\frac{\V(u_n,\clint{0,T})}{T}\right)^2\de\De\ell(t) \notag \\
 & = \int_{\clint{0,T}} \left(\frac{\V(u,\clint{0,T})}{T}\right)^2\de\De\ell(t) \notag \\
 & = \int_{\clint{0,T}}  \norm{h(t)}{}^2\de\De\ell(t), \notag
\end{align}
hence
\begin{equation}\label{norm to norm}
  \lim_{n \to \infty} \norm{h_n}{L^2(\De\ell;\H)}^2 = \norm{h}{L^2(\De\ell;\H)}^2.
\end{equation}
Now let us observe that from Lemma \ref{L:Q(v)}, formula \eqref{DQ = Q'Dl-2}, and from 
\eqref{def w wn} we have that
\begin{equation}
  \De w = h_n\De\ell, \qquad \De w_n = h_n \De \ell_n.
\end{equation}
Let us also recall that the vector space of (vector) measures $\nu : \borel(\clint{0,T}) \function \H$ can be endowed with the complete norm $\norm{\nu}{} := \vartot{\nu}(\clint{0,T})$, where 
$\vartot{\nu}$ is the total variation measure of $\nu$.
Moreover from the definition of variation, inequality \eqref{V(f)-V(g)}, and the triangle inequality, 
we infer that
\begin{equation}\label{Dl -Dln}
 \norm{\De\ell - \De\ell_n}{} = \vartot{\De\ \!(\ell - \ell_n)}(\clint{0,T}) = \V(\ell - \ell_n, \clint{0,T}) \to 0 \qquad \text{as $n \to \infty$}.
\end{equation}
From \eqref{|g_n|infty bdd} it follows that
\begin{equation}
  \vartot{\De w_n}(B) = \int_B \norm{h_{n}(t)}{} \de \De\ell(t) \le C \vartot{\De \ell_n}(B), \qquad
  \forall B \in \borel(\clint{0,T}),
\end{equation}
therefore, since $\De\ell_n \to \De \ell$ in the space of real measures, we infer that for every $\eps > 0$ there exists $\delta > 0$ such that 
\[
  \vartot{\De \ell}(B) < \delta \ \Longrightarrow\ \sup_{n \in \en} \vartot{\De w_n}(B) < \eps
\]
for every $B \in \borel(\clint{0,T})$. This allows us to apply 
the weak sequential compactness Dunford-Pettis theorem for vector measures (cf. \cite[Theorem 5, p. 105, Theorem 1, p. 101]{DieUhl77} ) and we deduce that, at least for a subsequence, 
$\De w_n$ is weakly convergent to some measure $\nu : \borel(\clint{0,T}) \function \H$. 
Hence thanks to \eqref{wn->w unif} and to \cite[Lemma 7.1]{KopRec16} we infer that 
\begin{equation}
  \De w_n \ \text{is weakly convergent to}\  \De w,
\end{equation}
in particular for every bounded Borel function $\varphi : \clint{0,T} \function \H$, we have that  the functional 
\newline $\nu$ $\longmapsto$ $\int_{\clint{0,T}} \duality{\varphi(t)}{\de \nu(t)}$ is linear and continuous on the space of measures with bounded variation and we have
\[
  \lim_{n \to \infty} \int_{\clint{0,T}} \duality{\varphi(t)}{\de \De w_n(t)} = 
  \int_{\clint{0,T}} \duality{\varphi(t)}{\de \De w(t)},
\]
that is
\begin{equation}\label{gn Dwn -> gDw}
  \lim_{n \to \infty} \int_{\clint{0,T}} \duality{\varphi(t)}{h_{n}(t)} \de \De \ell_n(t) = 
  \int_{\clint{0,T}} \duality{\varphi(t)}{h(t)}\de \De \ell(t).
\end{equation}
On the other hand, by \eqref{|g_n|infty bdd} there exists $\eta \in \L^{2}(\De\ell;\H)$ such that 
$h_{n}$ is weakly convergent to $\eta$ in $\L^2(\De\ell;\H)$, therefore if we set 
$\psi_n(t) := \duality{\varphi(t)}{h_{n}(t)}$ and $\psi(t) := \duality{\varphi(t)}{\eta(t)}$ for $t \in \clint{0,T}$, we have that 
$\psi_n$ is weakly convergent to $\psi$ in $\L^2(\De\ell;\re)$, and
\begin{align}
   & \sp \left| \int_{\clint{0,T}} \psi_n(t) \de \De \ell_n(t) -  \int_{\clint{0,T}} \psi(t) \de \De \ell(t) \right| \notag \\
   &  \le  \int_{\clint{0,T}} |\psi_n(t)| \de \vartot{\De\ \!(\ell_n - \ell)}(t) +
       \left| \int_{\clint{0,T}} (\psi_n(t) - \psi(t)) \de \De \ell(t) \right| \notag \\
   & \le  \norm{\varphi}{\infty}\norm{h_{n}}{\infty} \vartot{\De\ \!(\ell_n - \ell)}(\clint{0,T}) +
       \left| \int_{\clint{0,T}} (\psi_n(t) - \psi(t)) \de \De \ell(t) \right| \to 0 
\end{align}
as $n \to \infty$, because \eqref{|g_n|infty bdd} and \eqref{Dl -Dln} hold, and $\psi_n$ is weakly convergent to $\psi$ in $\L^2(\De\ell;\re)$. Therefore we have found that 
\[
  \lim_{n \to \infty} \int_{\clint{0,T}} \duality{\varphi(t)}{h_{n}(t)} \de \De \ell_n(t)\ = 
  \int_{\clint{0,T}} \duality{\varphi(t)}{\eta(t)}\de \De \ell(t),
\]
hence, by \eqref{gn Dwn -> gDw},
\begin{equation}\label{g dl = z dl weakly}
  \int_{\clint{0,T}} \duality{\varphi(t)}{\de(h\De \ell)(t)} = 
  \int_{\clint{0,T}} \duality{\varphi(t)}{\de (\eta\De \ell)(t)}.
\end{equation}
The arbitrariness of $\varphi$ and \eqref{g dl = z dl weakly} implies that $\eta\De \ell = h\De \ell$ (cf. 
\cite[Proposition 35, p. 326]{Din67}), hence $\eta(t) = h(t)$ for $\De\ell$-a.e. $t \in \clint{0,T}$ and we have found that
\begin{equation}\label{gn to g deb}
  h_{n} \convergedeb h \qquad \text{in $\L^2(\De\ell;\H)$}. 
\end{equation}
Since $\L^2(\De\ell;\H)$ is a Hilbert space, from \eqref{norm to norm} and \eqref{gn to g deb}
we deduce that
\begin{equation}
 h_{n}  \to h \qquad \text{in $\L^2(\De\ell;\H)$} ,
\end{equation}
and, since $\De\ell(\clint{0,T})$ is finite,
\begin{equation}
  h_{n} \to h \qquad \text{in $\L^1(\De\ell;\H)$}.
\end{equation}
Hence, at least for a subsequence which we do not relabel, $h_n(t) \to h(t)$ for $\De\ell$-a.e. $t \in \clint{0,T}$, thus
\begin{align}
 V(w_n - w, [0,T])=  \norm{\De\ \!(w_n - w)}{}
    & =  \norm{\De w_n - \De w}{} =
            \norm{h_{n} \De \ell_n - h \De \ell}{} \notag \\
    & \le \norm{h_{n}\De\ \!(\ell_n - \ell)}{} + \norm{(h_{n} -h) \De\ell}{} \notag \\
    & \le C \norm{\De\ \!(\ell_n - \ell)}{} + 
            \int_{\clint{0,T}}\norm{h_{n}(t) - h(t)}{} \de \De\ell(t) \to 0\notag
\end{align}
as $n \to \infty$ and we have proved that $\norm{w - w_n}{\BV} \to 0$ as $n\to \infty$. We can conclude recalling \eqref{def w wn} and that $\Q(v) = 2\Pl(v) - v$ for every 
$v \in \Czero(\clint{0,T};\H) \cap \BV(\clint{0,T};\H)$.
\end{proof}

We can finally infer the strict continuity of the play operator on 
$\Czero(\clint{0,T};\H) \cap \BV(\clint{0,T};\H)$.

\begin{proof}[Proof of Theorem \ref{T:BVstrict cont}]
The proof of Theorem \ref{T:BVstrict cont} is now a consequence of Theorem \ref{T:BVnorm cont}
and \cite[Theorem 3.4]{Rec11}.
\end{proof}


\end{document}